\documentclass[11pt]{amsart}
\usepackage{graphicx}

\usepackage[active]{srcltx}
 \makeatletter
\renewcommand*\subjclass[2][2000]{%
  \def\@subjclass{#2}%
  \@ifundefined{subjclassname@#1}{%
    \ClassWarning{\@classname}{Unknown edition (#1) of Mathematics
      Subject Classification; using '1991'.}%
  }{%
    \@xp\let\@xp\subjclassname\csname subjclassname@#1\endcsname
  }%
}
 \makeatother
\usepackage{enumerate,url,amssymb,  mathrsfs,yhmath}

\newtheorem{theorem}{Theorem}[section]

\newtheorem*{lemma*}{Lemma}
\newtheorem{proposition}[theorem]{Proposition}

\def\1ton{1,2,\ldots,n}

\usepackage{amssymb}
\usepackage{amsthm}
\usepackage{mathrsfs, amsfonts, amsmath}
\usepackage{graphicx}
\usepackage{psfrag}
\usepackage{color}

\newcommand{\bydef}{\stackrel{{\rm def}}{=\!\!=}}

\theoremstyle{definition}

\theoremstyle{remark}

\numberwithin{equation}{section}

\def\XXint#1#2#3{{\setbox0=\hbox{$#1{#2#3}{\int}$}
\vcenter{\hbox{$#2#3$}}\kern-.5\wd0}}

\def\le{\leqslant}
\def\ge{\geqslant}
\begin{document}

\title{Total energy of radial mappings} \subjclass{Primary 31A05;
Secondary 42B30 }


\author{Shaolin Chen}
\address{S. L. Chen, College of Mathematics and Statistics, Hengyang Normal University, Hengyang, Hunan 421008,
People's Republic of China.} \email{mathechen@126.com}

\author{David Kalaj}
\address{D. Kalaj, University of Montenegro, Faculty of Natural Sciences and
Mathematics, Cetinjski put b.b. 81000 Podgorica, Montenegro}
\email{davidk@ac.me}

\begin{abstract}

 We prove that, the so called total energy functional defined on the class
 of radial streachings between annuli attains its minimum on a total energy diffeomorphism between annuli on $\mathbf{R}^n$. This involves a subtle analysis of some special ODE. The result is an extension of the corresponding $2-$dimensional case obtained by Iwaniec and Onninen (Arch. Ration. Mech. Anal., 194: 927-986, 2009).
\end{abstract}  \maketitle


\section{Introduction}

\subsection{Total energy}

Assume that $h\in \mathscr{W}^{1,n}$ is a homeomorphism between two
annuli $\mathbb{A}=A(r,R)$ and $\mathbb{A}_\ast=A(r_*, R_*)$ of the
Euclidean space $\mathbf{R}^n$. Then the total energy of $h$ is
defined by Iwaniec and Onninen in \cite{arma} by the formula
$$\mathcal{E}[h]=\frac{\alpha}{|\mathbb{A}_\ast|}\int_{\mathbb{A}} \|Dh\|^n +\frac{\beta}{|\mathbb{A}|}\int_{\mathbb{A}_\ast}\|D{h^{-1}}\|^n,$$ $\alpha+\beta=1$, $
\alpha>0,$ $\beta>0$.
The functional $$h\rightarrow \int_{\mathbb{A}} \|Dh\|^n$$ is called the $n-$energy functional while the functional
$$h\rightarrow \int_{\mathbb{A}_\ast}\|D{h^{-1}}\|^n$$ is called the distortion functional.
We define a radial stretching $h$ as a mapping defined by a homeomorphism $H:[r,R]\mapsto [r_\ast, R_\ast]$ so that $$h(x)=H(|x|)\frac{x}{|x|}.$$

In \cite{arma}, Iwaniec and Onninen  showed that the minimum of
total energy for $n=2$ attained by a stretching diffeomorphism of
$\mathbb{A}$ onto $\mathbb{A}_\ast$. One of their key steps was to
solve the principal solution of so called equilibrium equation for
the radial mappings. It is the following boundary value problem
\begin{equation}\label{2dim}\left\{
  \begin{array}{ll}
    \ddot H=(H-t \dot H)\frac{(\alpha H\dot H+\beta t)\dot H^2}{(\alpha t \dot H^3+\beta t)tH} & \hbox{} \\
    H(r)=r_\ast,\ \ \ H(R)=R_\ast. & \hbox{}
  \end{array}
\right.\end{equation}

Namely they proved the following theorem.
\begin{theorem}$($\cite[Theorem 5.~1]{arma}$)$\label{main}
Given $R>r>0$ and $R_\ast>r_\ast>0$ there exists an unique strictly
increasing function $H\in \mathscr{C}^\infty[r,R]$ that solves the
equation \eqref{2dim} such that $H[r,R]=[r_\ast, R_\ast]$.
\end{theorem}

Furthermore,  in one of their main results
(\cite[Theorem~1.4]{arma}), they proved that the mapping
$h(x)=H(|x|)\frac{x}{|x|}$ is the minimizer of the total energy
functional (for $n=2$).
 In \cite[Theorem~1.5]{arma}, they showed that this result cannot be extended to the Euclidean space $\mathbb{R}^n$ for $n\ge 4$, however the case $n=3$ remains an open problem.

In this paper, we extend Theorem~\ref{main} by proving the following
theorem.

\begin{theorem}\label{wienn}
Let $n\ge 3$. If $\mathcal{P}(A,A_*)$ is the family of radial mappings with finite total energy, then there is a  radial diffeomorphism $h=h_{\lambda_*}$ that minimizes
the  functional of total energy $\mathcal{E}:\mathcal{P}(A,A_*)\to \mathbf{R}$.
\end{theorem}
The total energy is indeed a linear combination of the two operators, the energy functional and distortion functional. However it turns out that to minimize separately
 those two functionals do not solve the combination problem (\cite{arma}).
The problem of finding a minimizer throughout certain class of
homeomorphism has a long history. We want to refer here to some
recent paper concerning minimization problem of harmonic Euclidean
energy \cite{nconj,invent} and of non-Euclidean energy
\cite{calculus} of heomorphisms between given annuli on Euclidean
plane and on a Riemannian space respectively. Further, for
minimization problem of distortion functional, we refer to the
papers  \cite{astala} and \cite{klondon}. On the generalization of
those problem for the spatial annuli and for $n$-harmonic energy
(respectively $(\rho,n)$) energy, see the papers \cite{memoirs} and
\cite{kalajarxiv}.
\section{The proof of main result}
\subsection{Hilbert norm of derivatives of the radial stretching and of its inverse}
Assume that $h(x)=H(s)\frac{x}{s}$, where $s=|x|$. Let
$\mathcal{H}(s)\bydef\frac{H(s)}{s}$. Since $\mathrm{grad}
(s)=\frac{x}{|x|}$, we obtain
$$Dh(x)=\left(\mathcal{H}(s)\right)'\frac{x\otimes   x}{s}+ \mathcal{H}(s) \mathbf{I},$$ where $\mathbf{I}$ is the identity matrix.
For $x\in {\mathbb A}$, let $T_1=N=\frac{x}{|x|}$. Further, let
$T_2$, $\dots$, $T_n$ be $n-1$ unit vectors mutually orthogonal and
orthogonal to $N$. Thus
\[\begin{split} \|Dh(x)\|^2&=\sum_{i=1}^n|Dh(x) T_i |^2=\sum_{i=1}^n\left|\left(\mathcal{H}(s)\right)' \frac{\left<T_i,x\right>}{s} x +\mathcal{H}(s) T_i \right|^2
\\&
= \left(\mathcal{H}'(s)\right)^2  s^2+n \left( \mathcal{H}(s)\right)^2+2\mathcal{H}(s)\mathcal{H}'(s) s
 \\&=\frac{n-1}{s^2} H^2+\dot H^2.\end{split}\]
Moreover, with respect to the basis $T_i$ ($i=1,\dots, n$), we have
$$D^*h Dh=\left(
                                   \begin{array}{cccc}
                                     \dot H^2 & 0 & \dots & 0 \\
                                     0 & \frac{H^2}{s^2} & \dots & 0 \\
                                     \vdots &  \vdots & \ddots & 0 \\
                                     0 & 0 & \dots & \frac{H^2}{s^2} \\
                                   \end{array}
                                 \right).$$ Here $D^*h$ is the adjugate of the matrix $Dh$.
Thus $$J=J_h=\frac{\dot HH^{n-1}}{s^{n-1}}.$$ Let $X=Dh^{-1}(h(x))$.
Then
$$X X^* = (D^*hDh)^{-1}=  \left(
                                   \begin{array}{cccc}
                                     \dot H^{-2} & 0 & \dots & 0 \\
                                     0 & \frac{s^2}{H^2} & \dots & 0 \\
                                     \vdots &  \vdots & \ddots & 0 \\
                                     0 & 0 & \dots & \frac{s^2}{H^2} \\
                                   \end{array}
                                 \right),$$
which implies that
$$\|X\|= \sqrt{\dot H^{-2}+\frac{(n-1)s^2}{H^2}}.$$

Hence $$\|Dh^{-1}(h(x))\|^n J_h =J_h\left[\dot
H^{-2}+\frac{(n-1)s^2}{H^2}\right]^{\frac{n}{2}}.$$
\subsection{Total energy of radial mappings}
Let $a=\frac{\alpha}{|\mathbb{A}_\ast|}$ and $b=\frac{\beta}{|\mathbb{A}|}$. Then we calculate the energy of a radial stretching $h=H(r)\frac{x}{|x|}$.
We obtain \[\begin{split}\mathcal{E}[h]&=a\int_{\mathbb{A}} \|Dh\|^n +b\int_{\mathbb{A}_\ast}\|D{h^{-1}}\|^n= \omega_{n-1} \int_r^R \mathcal{L}[s,H,\dot H] ds,\end{split}\] where

\[\begin{split}\mathcal{L}[s,H,\dot H]&=as^{n-1} \|Dh(x)\|^{n}+b\|Dh^{-1}(h(x))\|^{n}J_h
\\&=as^{n-1} \left[\frac{(n-1) H^2}{s^2}+\dot H^2\right]^{\frac{n}{2}}+bH^{n-1} \left[\frac{(n-1) s^2}{H^2}+\frac{1}{\dot H^2}\right]^{\frac{n}{2}} \dot H\end{split}\]
and  $\omega_{n-1}$ is the Hausdoff measure of the unit sphere. Then
the equilibrium equation (Euler-Lagrange equation) is
$$\mathcal{L}_H=\partial_s \mathcal{L}_{\dot H}$$ and it  reduces to
the equation
\begin{equation}\label{convex}\ddot H=(s\dot H-H) M(s),\end{equation} where

\begin{equation}\label{M}M(s)=\frac{\mathrm{I}+\mathrm{II}}{\mathrm{III}},\end{equation}

$$ \mathrm{I}= a{s^{-1} \left[{(n-1) H^2}+{s^2}\dot H^2\right]^\frac{n-4}{2}
\left[(n-1) H^2+(n-2) s H \dot H+s^2 \dot H^2\right]},$$
$$\mathrm{II}=b{H^{n-1} \left[(n-1)\dot H^2 s^2+H^2\right]^\frac{n-4}{2} \dot
H^{1-n}\left[H^2+(n-2) s H \dot H+(n-1) s^2 \dot H^2\right]} $$ and

\begin{eqnarray*}
\mathrm{III}&=&\left(H^2+s^2 \dot H^2\right) \Bigg\{a{s \big[{(n-1) H^2}+\dot
H^2{s^2}\big]^\frac{n-4}{2}}\\
&& +b\frac{H \left[(n-1) s^2{\dot
H^2}+{H^2}\right]^\frac{n-4}{2}}{\dot H^{1+n} }\Bigg\}.
\end{eqnarray*} So $$ (H-s\dot H)'=-s (H-s\dot H) M(s),$$ which is
equivalent to

$$\big[\log(H-s\dot H)\big]'=-s M(s).$$
Thus $$\log(H-t\dot H) = \int_s^{s_1} (-\tau M(\tau))d\tau+c,$$
which gives that

\begin{equation}\label{lcon}H-s\dot H=c \exp[\int_{s_1}^{s} (-\tau M(\tau))d\tau].\end{equation}

Now we consider the following boundary problem 

\begin{equation}\label{ndim}\left\{
  \begin{array}{ll}
    \ddot H=(s\dot H-H) M(s) & \hbox{} \\
    H(r)=r_\ast,\ \ \ H(R)=R_\ast. & \hbox{}
  \end{array}
\right.\end{equation} which is the $n-$dimensional generalization of the boundary problem \eqref{2dim}.

Now we prove that the diffeomorphic  solution of \eqref{ndim} does
exist. The idea is simple, we want to reduce the equation
\eqref{ndim} into
 an ODE of the first order, but to do this we assume that the diffeomorphic solution $H$ exists. This assumption is not harmful. Namely,
  the proof can be started from a certain first order ODE \begin{equation}\label{helpi}F'=G[t,F(t)],\end{equation} which has to do nothing
   with $H$ (see \eqref{aux} below). Then we solve \eqref{helpi} and, by using the solutions of it, we construct solutions of \eqref{ndim}.
   Such a solution $H$ will be a diffeomorphism and so it  will satisfy one of the three cases listed below. On the other hand if we have a diffeomorphic solution $H$ of \eqref{ndim}, then it will satisfy the equation \eqref{lcon} for some continuous $M$ and this will imply the uniqueness of solution $H$. 

So if $H$ is a strictly increasing $C^2$ diffeomorphism defined in a domain $(a,b)$ that solves the equation \eqref{ndim},
then $\dot H(s)>0$ and from \eqref{lcon}, we conclude that there are three possible cases:
\begin{itemize}
\item {\bf Case~1} $c=0$. Then $H-s\dot H\equiv 0$, or what is the same $H(s)=c s$, and this produces a linear mapping $h(x)=cx$, so in this case $$\frac{R}{r}=\frac{R_\ast}{r_\ast}.$$
\item {\bf Case~2} $c>0$. Then $H-s\dot H>0$ and thus $t(s)=\frac{H(s)}{s}$ is a monotone increasing.
\item {\bf Case~3} $c<0$. Then $H-s\dot H<0$ and thus $t(s)=\frac{H(s)}{s}$ is a monotone decreasing.
\end{itemize}

If $$\frac{R}{r}\neq \frac{R_\ast}{r_\ast},$$ we  take the new variable $t=\frac{H(s)}{s}$ and the new function
\begin{equation}\label{prin}F(t)=\dot H\left(\frac{H(s(t))}{s(t)}\right),\end{equation} where $s(t)$ is the inverse of $t=t(s)$. Then we obtain $$\ddot H(s)=\frac{\dot F(t) (s \dot H - H)}{s^2}$$
 and

\begin{equation}\label{aux}\dot F(t)=G[t,F(t)]=-\frac{U[t,F(t)]+V[t,F(t)]}{W[t,F(t)]},\end{equation} where for $(t,y)\in \mathbf{R}_+^2$

$$U(t,y)\bydef a\left[(n-1) t^2+y^2\right]^{\frac{n-4}{2}} \left[(n-1) t^2+(n-2) t y+y^2\right], $$

$$V(t,y)\bydef b\frac{\left[(n-1) y^2+{t^2}\right]^{{\frac{n-4}{2}}}  \left[t^2+(n-2) t y+(n-1) y^2\right]}{ty^{n-1}},$$
and
$$W(t,y)\bydef\left(t^2+y^2\right) \left\{a\left[(n-1) t^2+y^2\right]^{\frac{n-4}{2}}+b\frac{t \left[(n-1){y^2}+{t^2}\right]^{{\frac{n-4}{2}}}}{y^{1+n} }\right\}.$$
Observe that $G:\mathbf{R}_+^2\to \mathbf{R}_{-}$ is smooth on
$\mathbf{R}_+^2$. Using the Picard-Lindel\"of theorem, we observe
that through any point $(t_\circ, y_\circ)\in\mathbf{R}_+^2$ there
passes exactly one smooth integral curve defined in a neighborhood
of this point. Then the extension theorem, see  \cite{hartman},
tells us that such a local solution   extends uniquely (as a
solution) to the so-called maximal interval of existence. Denote
this interval by $(\alpha, \beta)$, where $0 \le \alpha< \beta\le
\infty$.

The characteristic feature of the maximal interval of existence is that the points
$(t, F(t))$ hit the boundary of $\mathbf{R}_+^2$ as $t \to  \alpha$ or $t \to  \beta$. Precisely, this means
that the points $(t, F(t))$ lie outside any given compact subset of $\mathbf{R}_+^2$ when
$t$ approaches $\alpha$ or $\beta$.
\begin{proposition}\label{propa}
Every local solution $F$ of the equation \eqref{aux} has
$(0,+\infty)$ as the interval of its existence. Moreover, $F$ is
decreasing and

\begin{itemize}
    \item $\lim_{t\to 0}F(t)=+\infty$ and
  \item $\lim_{t\to +\infty}F(t)=0$.
\end{itemize}
\end{proposition}
\begin{proof}
Since $G$ is negative, we see that $F$ is a decreasing function on
its maximal interval $(\alpha,\beta)$. Let $A=\lim_{t\downarrow
\alpha} F(t)$ and $B=\lim_{t\uparrow\beta} F(t).$ We will show that
$\alpha=0$ and $\beta=\infty$.
 First, we prove $\beta=\infty$. Since $B<A$, we see that $B\in\mathbf{R}_+$ and thus the point $(\beta,B)$
is a point
 of continuity of $G$ which means that $F$ can be continued smoothly above $\beta$. Hence $\beta=\infty$.

Now we prove $\alpha=0$. If we assume $\alpha>0$, then this
assumption will lead to contradiction. Since    $F$ is decreasing,
we know that $\lim_{t\downarrow \alpha}F(t)=+\infty$ and
$$\lim_{t\downarrow \alpha} F'(t)=\lim_{t\downarrow \alpha}
G(t,F(t))=-1.$$ But this is impossible, because  the sequence
$$y_m=F(\alpha+1)-F(\alpha+1/m)=\int_{\alpha+\frac{1}{m}}^{\alpha+1} F'(t)dt,$$ would be bounded.

Now we begin to show that $A=\infty$ and $B=0$.

If $A=F(0)\bydef \lim_{t\to 0}F(t)<\infty$, then, by \eqref{aux}, we obtain
\begin{equation}\label{tFt}\lim_{t\to 0}(t
F'(t))=-\frac{b}{a}(n-1)^{-1+\frac{n}{2}} A^{1-n}.\end{equation}
Then there is $\delta>0$ so that $0<t<\delta$ implies $$t F'(t)\le
C=-\frac{b}{2a}(n-1)^{-1+\frac{n}{2}} A^{1-n}.$$ Thus
$$F(\delta)-F(t)=\int_{t}^\delta F'(t)\le C \log \left(\frac{\delta}{t}\right),$$
and hence $$F(t)\ge F(\delta) -C \log \left(\frac{\delta}{t}\right)$$ implying that $\lim_{t\to 0} F(t)=\infty$ which is a contradiction.
 Similarly, by using   \eqref{aux}, we get \begin{equation}\label{tFt2}\lim_{t\to \infty }(t F'(t))=-\frac{a}{b}(n-1)^{-1+\frac{n}{2}} B^{n+1},
 \end{equation} where $$B=\lim_{t\to +\infty}F(t)$$ and in similar way  we establish the second statement.
\end{proof}

Observe that the graph of the solution $F$ intersects the diagonal $\{(x,x): x>0\}$ at exactly one point. Then we define the particular solution   $F=F_\lambda$,
with the initial condition $(\lambda,\lambda)\in \mathrm{Graf}(F)$. 

Now we prove the following
\begin{proposition}\label{kprop}
For fixed $t\in(0,\infty)$, the function $Q(\lambda)=F_\lambda(t)$ is an increasing $C^1$ function of $(0,\infty)$ onto itself.
\end{proposition}
\begin{proof}
The fact that $Q(\lambda)$ is of class $C^1$ follows from the theorem on dependence of initial conditions and parameters (\cite[Corollary~4.1.]{hartman}).

Further, for two different $\lambda_1<\lambda_2$, let
$R(t)=F_{\lambda_2}(t)-F_{\lambda_1}(t)$. Then $R(t)\neq 0$, because
near $t$, the Cauchy problem $F'(t)=G(t,F(t))$, $F(t)=F_0$ has the
unique solution. Thus $R(t)$ has the constant sign.  Further,
because $F_\lambda$ is decreasing,
$F_{\lambda_2}(\lambda_2)=\lambda_2>\lambda_1=F_{\lambda_1}(\lambda_1)>F_{\lambda_1}(\lambda_2)$,
it follows that $R(t)>0$. Thus $Q$ is increasing. In order to prove
that $\lim_{\lambda\downarrow 0} Q(\lambda)=0$ and
$\lim_{\lambda\uparrow 0} Q(\lambda)=\infty$ do as follows. For
$\epsilon<t<\epsilon^{-1}$, we obtain
$\epsilon=F_\epsilon(\epsilon)\ge F_\epsilon(t)$ and
$1/\epsilon=F_{1/\epsilon}(1/\epsilon)\le F_{1/\epsilon}(t)$,
because $t\to F_\lambda(t)$ is decreasing. This implies the
proposition
\end{proof}
Now we prove the following theorem which asserts that our boundary problem \eqref{ndim} has a unique diffeomorphic solution.

\begin{theorem}
Let $0<r<R$ and $0<r_*<R_*$. Then there is an increasing diffeomorphism $H:[r,R]\to [r_*,R_*]$ that  solves the principal equation \eqref{convex}.
\end{theorem}
\begin{proof}
According to \eqref{prin} and \eqref{aux}, the equation \eqref{convex} is reduced to solving the initial value problems
\begin{equation}\label{first}
\left\{
  \begin{array}{ll}
    H'(s)=F_\lambda(\frac{H(s)}{s}), & \hbox{for $s>0$ and $H(s)>0$;} \\
    H(r)=r_*, & \hbox{}
  \end{array}
\right.
\end{equation}
for $\lambda>0$. Further we find $\lambda>0$ so that $H(R)=R_*$.
Using the Picard-Lindel\"of theorem, we observe that for fixed $\lambda_\circ$ there is  exactly one smooth solution $H=H_{\lambda_\circ}$ of the
problem \eqref{first}. Let $(\alpha,\beta)$ be the maximal interval of existence of $H$. Similarly as in the proof of Proposition~\ref{propa},
we obtain that $\beta=\infty$. Namely if $\beta<\infty$, then since $F_{\lambda_\circ}$ is positive we conclude that
$B=H(\beta^-):=\lim_{s\uparrow \beta}\in [r_\ast,\infty]$. If $B=\infty$, then $H'(\beta^-)=F_{\lambda_\circ}(+\infty)=0$, and thus by \eqref{first},
$H(\beta^-)\le C\beta$ which is a contradiction. Thus $B<\infty$. But then $(x,y)=(\beta, B)$  is a point of continuity of $F_{\lambda_\circ}(\frac{y}{x})$,
so the solution $H$ can be extended above $\beta$. This implies that $\beta=\infty$. Since $r\in (\alpha,\beta)$, then $\alpha< r$. From \eqref{first} we obtain
$$H_\lambda(m)-r_\ast=\int_r^{m}F_\lambda\left(\frac{H_\lambda(s)}{s}\right) ds, \ \ m>r$$

Assume  that  $H_{\lambda_n}(m)<M$ and $\lambda_n \to \infty$. Since
$F_\lambda$ is positive, we conclude that  $H_{\lambda_n}$ is
increasing and
 $$\frac{H_{\lambda_n}(s)}{s} \le \frac{M}{r},\ \ \ r\le s\le m.$$ Thus $$H_{\lambda_n}(m)-r_\ast\ge \int_{r}^m F_{\lambda_n}\left(\frac{M}{r}\right) ds.$$
 The right hand-side  of the last inequality tends to $+\infty$ because of Proposition~\ref{kprop}. Thus $H_\lambda(m)$ tends to infinity when $\lambda\to\infty$.

Therefore, there is exactly one $\lambda_*$, such that
$H=H_{\lambda_*}$ maps the interval $[r,R]$ onto $[r_*, R_*]$.
\end{proof}

Next we prove Theorem~\ref{wienn}, which is equivalent to the
following theorem.
\begin{theorem}
Let $n\ge 2$. If $\mathcal{P}(A,A_*)$ is the family of radial
stretching   with finite total energy, then the radial mapping
$h=h_{\lambda_*}$, defined by
$h(x)=H_{\lambda_*}(|x|)\frac{x}{|x|}$,  minimizes the functional
$\mathcal{E}:\mathcal{P}(A,A_*)\to \mathbf{R}$.
\end{theorem}
\begin{proof}
Since we find the stationary point, which is unique, we only need to
show that the given energy integral attains its minimum. First, we
show that the function
\[\begin{split}\mathcal{L}[s,H,\dot H]&=as^{n-1} \left[\frac{(n-1) H^2}{s^2}+\dot H^2\right]^{\frac{n}{2}}\\&+bH^{n-1}
 \left[\frac{(n-1) s^2}{H^2}+\frac{1}{\dot H^2}\right]^{\frac{n}{2}} \dot H\end{split}\]
is convex in $\dot H$.

For $K=\dot H$ we have the following
 \[\begin{split}\partial_K\mathcal{L}[s,H,K]&=a K n \left[K^2+\frac{H^2 (n-1)}{s^2}\right]^{\frac{1}{2} (n-2)} s^{n-1}\\&-\frac{b H^{n-1} n
  \left[\frac{1}{K^2}+\frac{(n-1) s^2}{H^2}\right]^{\frac{1}{2} (n-2)}}{K^2}+b H^{n-1} \left[\frac{1}{K^2}+\frac{(n-1) s^2}{H^2}\right]^{\frac{n}{2}} \end{split}\]

and \[\begin{split}\partial^2_{KK}\mathcal{L}[s,H,K]&=(n-1) n
\left(K^2 s^2+H^2\right)\\&\times
\left\{a\frac{\left[H^2+\frac{(n-1)
s^2{K^2}}{H^2}\right]^\frac{n-4}{2} H}{K^{1+n}}+b{s
 \left[s^2K^2+{(n-1) H^2}\right]^\frac{n-4}{2}}\right\}\end{split}\] which is clearly positive. Further we can find a positive constant $C$ so that
\begin{equation}\label{bound1}C\left(|\dot H|^n+\frac{1}{|\dot H|^{n-1}}\right)\le \mathcal{L}[s,H,\dot H],\end{equation} which implies that the function $L$ is coercive.

  Let $h_m(x)=H_m[|x|]x/|x|$ be a sequence of smooth mappings  with $H_m(r)=r_\ast$, $H_m(R)=R_\ast$ and $$\inf_{h\in \mathcal{P}(A,A_*)}
  \mathcal{E}[h]=\lim_{m\to \infty}\mathcal{E}[h_m].$$ Then $H_m$ are diffeomorphisms because $\dot H_m\neq 0$. Then up to a subsequence it
   converges to a monotone increasing function $H_\circ$.  Moreover,  since $H_m$ is a bounded sequence of $\mathscr{W}^{1,n}$, it converges,
   up to a subsequence weakly to a mapping $H_\circ \in \mathscr{W}^{1,n}$.

 By using the mentioned convexity of $\mathcal{L}$  and the fact that $\mathcal{L}$ is coercive, by standard theorem from the calculus of
 variation (as in the proof of \cite[Theorem~1.2]{arma} (cf. \cite[Theorem~63]{gomes})),  we obtain that $$\mathcal{E}[h_\circ]= \lim_{m\to \infty}\mathcal{E}[h_m].$$

 Further as $\mathcal{L}[s,H,K]\in C^\infty(\mathbf{R}_+^3)$, with $\partial^2_{KK}\mathcal{L}[s,H,K]>0$, we infer that
 $H_\circ\in C^\infty[r,R]$ (see \cite[p.~17]{jost}) and $H_\circ$ is the solution of our Euler-Lagrange equation. Thus it coincides with $H_{\lambda_\ast}$.
\end{proof}

\end{document}